\documentclass[12pt]{amsart}
\usepackage{amscd,amssymb,mathrsfs}
\usepackage{graphics}
\usepackage{psfrag}
\usepackage[dvipdf]{graphicx}
\newtheorem*{thmmain}{Main Theorem}
\newtheorem*{thm*}{Theorem}
\newtheorem{thm}{Theorem}[section]
\newtheorem{lemma}[thm]{Lemma}

\newtheorem{defn}[thm]{Definition}

\renewcommand{\bar}[1]{#1^{-1}}

\newcommand{\R}{\mathbb R}
\newcommand{\C}{\mathbb C}
\renewcommand{\o}{\mathsf}
\newcommand{\opn}{\operatorname}
\newcommand{\tr}{\mathsf{tr}}
\newcommand{\Inn}{\mathsf{Inn}}

\newcommand{\Fix}{\mathsf{Fix}}

\newcommand{\Ta}{\mathsf{Tame}}   
\renewcommand{\Pr}{\mathsf{Proper}} 
\newcommand{\Dd}{\mathsf{FourSided}}   

\newcommand{\lA}{\ell_A} 
\newcommand{\lB}{\ell_B} 
\newcommand{\lX}{\ell_X} 
\newcommand{\lY}{\ell_Y} 

\newcommand{\topsurf}{\Sigma} 
\newcommand{\surf}{S} 

\newcommand{\oto}{\o{O}(2,1)}
\newcommand{\soto}{{\o{SO}(2,1)^0}}
\newcommand{\SOTO}{{\o{SO}(2,1)}}

\newcommand{\rto}{\mathsf{V}}

\newcommand{\Eto}{\mathsf{E}} 
\newcommand{\hyp}{{\mathsf H}^2}

\newcommand{\LL}{\mathsf L} 
\newcommand{\Z}{\mathbb Z}

\newcommand{\Q}{\mathscr Q}
\newcommand{\idealquad}{\mathcal Q}

\newcommand{\CP}{{\mathcal C}}

\renewcommand{\H}{\mathcal H} 

\newcommand{\vx}{{\mathsf x}}

\newcommand{\vu}{{\mathsf u}}
\newcommand{\vv}{{\mathsf v}}

\newcommand{\xo}[1]{#1^0}
\newcommand{\xp}[1]{#1^+}
\newcommand{\xm}[1]{#1^-}
\newcommand{\xpm}[1]{#1^{\pm}}

\newcommand{\XM}{\bar{X}}

\newcommand{\Quad}[1]{\R_+\langle#1\rangle}

\newcommand{\HA}{\mathfrak{H}_A}
\newcommand{\HX}{\mathfrak{H}_X}
\newcommand{\HAM}{\bar{A}(\HA^c)}
\newcommand{\HXM}{\bar{X}(\HX^c)}

\newcommand{\infinity}{\infty}
\newcommand{\ldot}[2]{#1\cdot#2}

\newcommand{\G}{G}

\newcommand{\hp}{\mathsf{H}^2} 
\newcommand{\hthree}{\mathsf{H}^3}

\newcommand{\ZZ}{Z^1(\Gamma_0,\rto)}
\newcommand{\HH}{H^1(\Gamma_0,\rto)}

\newcommand{\rpt}{\R\mathsf{P}^2}
\newcommand{\basepoint}{o} 
\newcommand{\tbasepoint}{\tilde\basepoint} 
\newcommand{\fg}{\pi_1(\Sigma,\basepoint)} 
\newcommand{\tS}{\tilde{S}} 
\newcommand{\dev}{\mathsf{dev}}
\newcommand{\Isom}{\mathsf{Isom}}
\newcommand{\PGLtR}{\mathsf{PGL}(2,\R)}
\newcommand{\GLtR}{\mathsf{GL}(2,\R)}
\newcommand{\SL}{\mathsf{SL}}
\newcommand{\SLtC}{\SL(2,\C)}
\newcommand{\SLtR}{\SL(2,\R)}
\newcommand{\PSLtC}{\mathsf{PSL}(2,\C)}
\newcommand{\Fricke}{\mathfrak{F}(\topsurf)}
\newcommand{\csch}{\opn{csch}}

\newcommand{\trho}{\tilde{\rho}} 

\newcommand{\Hom}{\mathsf{Hom}}
\newcommand{\Ht}{\mathsf{H}^2}

\newcommand{\g}{g} 
\newcommand{\oa}[2]{\alpha_{#1}(#2)} 
\newcommand{\na}{\mu} 

\begin{document}

\title[Finite-sided deformation spaces]
{Finite-sided deformation spaces of complete affine 3-manifolds}
\author[Charette]{Virginie Charette}
\author[Drumm]{Todd A. Drumm}
\author[Goldman]{William M. Goldman}
\date{\today}  

\thanks{Charette gratefully acknowledges partial support from the
  Natural Sciences and Engineering Research Council of Canada.
  Goldman gratefully acknowledges partial support from National
  Science Foundation grant DMS070781.}

\begin{abstract}
A {\em Margulis spacetime\/} is a complete affine $3$-manifold $M$
with nonsolvable fundamental group. Associated to every Margulis
spacetime is a noncompact complete hyperbolic surface $\surf$.  We
show that every Margulis spacetime is orientable, even though $\surf$
may be nonorientable.  We classify Margulis spacetimes when $\surf$ is
homeomorphic to a two-holed cross-surface $\topsurf$, that is, the
complement of two disjoint discs in $\rpt$.  We show that every such
manifold is homeomorphic to a solid handlebody of genus two, and
admits a fundamental polyhedron bounded by crooked
planes. Furthermore, the deformation space is a bundle of convex
$4$-sided cones over the space of marked hyperbolic structures.
The sides of each  cone are defined by invariants of the
two components of $\partial\topsurf$ and the two orientation-reversing
simple curves.  The two-holed cross-surface, together with the
three-holed sphere, are the only topologies $\topsurf$ for which the
deformation space of complete affine structures is finite-sided.
\end{abstract}

\maketitle

\setcounter{tocdepth}{1} 
\tableofcontents
 
\section{Introduction}

A {\em Margulis spacetime\/} is a geodesically complete flat 
Lorentzian $3$-manifold $M^3$ with free fundamental group.
Such manifolds are quotients $\Eto/\Gamma$ of 
$3$-dimensional {\em Minkowski space\/} $\Eto$ by a discrete group $\Gamma$
of isometries acting properly on $\Eto$. 
By 
\cite{FriedGoldman,Mess} every complete affinely 
flat $3$-manifold 
has solvable fundamental group
or is a  Margulis spacetime.
In the latter case the linear holonomy 
$$
\pi_1(M) \cong \Gamma\xrightarrow{\LL} \SOTO
$$ 
is an embedding of $\Gamma$ onto a discrete subgroup
$\Gamma_0  = \LL(\Gamma)$.
Thus, associated to every Margulis spacetime is a complete hyperbolic
surface $\surf = \hp/\Gamma $. 
This hyperbolic surface has an intrinisic description:
it consists of parallelism classes of timelike lines 
(particles) in $M$.  
In particular $\Gamma$ is an {\em affine deformation\/} of the
Fuchsian group $\Gamma_0$, and we say that $M$ is an 
{\em affine deformation\/}
of the hyperbolic surface $\surf$. 

We conjecture that every Margulis spacetime $M$ is {\em tame},
that is, admits a polyhedral decomposition by crooked planes.
A consequence is that $M$ is {\em topologically tame,\/} that is,
homeomorphic to an open solid handlebody.
We start with the simplest groups of interest,
those whose holonomy is a rank two group.  (The cyclic case is
trivial.)  We previously established the conjecture when $\surf$ is a
three-holed sphere~\cite{CDG}.  In this paper, we establish the
conjecture for its nonorientable counterpart, when $\surf$ is
homeomorphic to a two-holed projective plane.  Every nonorientable
surface of negative Euler characteristic contains an embedded
two-holed cross-surface.  For this reason, we expect the two-holed
cross-surface, like the three-holed sphere, to act as a building block
for proving the general conjecture.

J.\ H.\ Conway has introduced the term {\em cross-surface\/} for the
topological space underlying the real projective plane $\rpt$.  Let
$\topsurf$ be a {\em $2$-holed cross-surface,\/} that is, the topological
surface underlying the complement of two disjoint discs in $\rpt$.
As is common in deformation theory, we fix the topology
of $\topsurf$ and consider {\em marked geometric structures,\/}
that is, homotopy classes of homeomorphisms $\topsurf\rightarrow\surf$,
where $\surf$ is a surface with a hyperbolic structure.

The first part of the paper focuses on the hyperbolic structure.  In
particular we describe the space of marked complete hyperbolic
structures on $\topsurf$.  We call this space the {\em Fricke space}
of $\topsurf$ and denote it by $\Fricke$.  In the second part of the
paper, we explicitly describe the space of all Margulis spacetimes
arising from $\topsurf$ as a bundle of convex $4$-sided cones over
$\Fricke$.

Like~\cite{CDG}, the classification involves both the topology of $M$ and the
geometry of the deformation space.
We show that every affine deformation of $\surf$ is homeomorphic to a solid
handlebody of genus two by constructing an explicit fundamental polyhedron.  
Although the two-holed projective plane is a nonorientable
surface, its affine deformations are orientable  $3$-manifolds.
In fact, every Margulis spacetime is orientable (Lemma~\ref{lem:orientable}).

There are four isotopy classes of 
essential primitive simple closed curves on $\topsurf$.
Two of these curves, denoted $A$ and $B$, 
correspond to components of $\partial\topsurf$. 
The other two, denoted $X$ and $Y$, 
reverse orientation and intersect transversely in one point.

If we fix a hyperbolic structure $\surf$, 
the boundary curves $A$ and $B$ correspond either to
closed geodesics bounding {\em funnels\/} (complete ends
of infinite area with cyclic fundamental group) or {\em cusps\/} 
(ends of finite area).
When a curve $K$ bounds a funnel, 
$\ell_K$ is the (positive) length of the closed geodesic and 
when the boundary curve cuts off a cusp $\ell_K$ is set equal to $0$. 
An orientation-reversing curve  $K$ always corresponds to a closed geodesic, 
so that $\ell_K$ represents this length and is always positive. 
For
$$
\big( \ell_A, \ell_B, (\ell_X,\ell_Y)\big) 
\;\in\;
\big( [0,\infinity)\times [0, \infinity) \big) \times \R_+^2
$$
the lengths satisfy the identity
\eqref{eq:LengthIdentity}, and the Fricke space $\Fricke$ 
identifies with 
$$
\big( [0,\infinity)\times [0, \infinity) \big) \times \R_+ .
$$
The boundary $\partial\Fricke$ has two faces sharing a common edge.
One face is defined by $\ell_A = 0$ and the other by $\ell_B = 0$.
These faces intersect in the line $\ell_A = \ell_B = 0$, which
corresponds to complete hyperbolic structures with finite area.

\medskip

Fix a linear representation $\rho_0$
of $\pi_1(S)$ in the special orthogonal group $\SOTO$.
An {\em affine deformation\/} of $\rho_0$  
is a representation $\rho$ in the group of Lorentzian isometries
whose linear part is $\rho_0$, that is, 
\begin{equation*}
\LL\circ\rho = \rho_0.
\end{equation*}
An affine deformation $\rho$ of $\rho_0$ 
is {\em proper\/} if it defines a proper affine action
of $\pi_1(S)$ on $\Eto$.
Equivalently $\rho$ is the holonomy representation of a complete
affine structure on a $3$-manifold $M^3$ and the isomorphism
$$
\pi_1(S) \longrightarrow \pi_1(M^3)
$$
induced by $\rho_0$ and $\rho$
arises from a homotopy-equivalence 
$ S \longrightarrow M^3$.

Affine deformations of the fixed linear representation $\rho_0$
correspond to {\em cocycles\/} and form a vector space $\ZZ$.
Slightly abusing notation, denote also by $\rto$ the $\Gamma_0$-module
given by the linear holonomy representation $\rho_0$ of $\Gamma_0$ on
the vector space $\rto$.  Translational conjugacy classes of affine
deformations form the {\em cohomology group\/}
$H^1\big(\Gamma_0,\rto\big)$.

When $\surf$ is a two-holed cross-surface,
or, more generally, any surface with $\pi_1(\Sigma)$ free
of rank two,  $\HH \cong \R^3$.
In his original work~\cite{Margulis1,Margulis2}, 
Margulis introduced a 
{\em marked signed Lorentzian length spectrum invariant,\/}
which now bears his name. This $\R$-valued class function on $\pi_1(\surf)$
detects the properness of an affine deformation, and determines an affine
deformation up to conjugacy~\cite{CharetteDrumm_isospectrality,DrummGoldman_isospectrality}.

Under the correspondence between
affine deformations and infinitesimal deformations of the hyperbolic structure on
$\hp$, Margulis's invariant identifies with the first derivative of the 
geodesic length function on Fricke space 
\cite{GoldmanMargulis,Goldman_MargulisInvariant}.
Originally defined for hyperbolic
isometries of $\Eto$, it has been extended to parabolic isometries 
in 
\cite{CharetteDrumm}, and to a function on geodesic currents 
in
\cite{GLM}. 

Given a curve $K\subset\surf$, the Margulis invariant $\mu_K$ is a linear
functional on $\HH$.  The four functionals associated to the curves
$A,B,X,Y$ satisfy the linear dependence \eqref{eq:alphaCDMrelations}.
Let $\Dd$ denote the open cone in $\HH$ such that all the quantities $
\na_A, \;\na_B , \;\na_X , \;\na_Y.  $ are either all positive or all
negative.  Figure~\ref{fig:DeformationSpace} offers a projective view
of $\Dd$; each line corresponds to a zero-set of a $\mu_K$ for some
curve $K$. 

Proper affine deformations will be investigated via 
{\em crooked  planes\/}, polyhedral surfaces introduced 
in
\cite{Drumm_thesis,Drumm_jdg,DrummGoldman_crooked}.  
A proper affine deformation is tame if it admits a fundamental
polyhedron bounded by crooked planes.  
The {\em Crooked Plane Conjecture\/} 
\cite{DrummGoldman_crooked,CDG}, which we prove for affine
deformations of the two-holed cross-surface, asserts that every
proper affine deformation is tame.

Denote the subspace of $\HH$ consisting of proper affine
deformations by $\Pr$ and the subset of $\Pr$ consisting 
of tame affine deformations by $\Ta$. 
Our main result is that $$\Ta \;=\; \Pr\;=\Dd.$$

%

\begin{thmmain}
Let $\surf$ be a two-holed cross-surface.  Then every proper affine deformation of $\surf$ is tame and the affine $3$-manifold is homeomorphic to a handlebody
of genus two.
\end{thmmain}

Let us underscore the fact that $\Pr$ is finite-sided, just like in the
case of the three-holed sphere, where proper deformations make up
a three-sided polyhedral cone~\cite{CDG}. In all other cases, the space of
proper affine deformations has infinitely many faces. 
In particular, for the other two surfaces of Euler characteristic
$ -1$ (or equivalently, surfaces with fundamental group of rank two), 
namely the one-holed torus and one-holed Klein bottle,
the space of proper affine deformations is defined by (necessarily)
infinitely many linear conditions~\cite{Charette_nonproper,CDGp}.

\subsection*{Acknowledgments}
We are grateful to Francis Bonahon, Suhyoung Choi, David Gabai, 
Fran\c cois Labourie, Grisha Margulis and  Yair Minsky, 
Ser-Peow Tan for helpful conversations. 
We also thank the referee for many useful suggestions.

\section{The two-holed cross-surface}\label{sec:thx}

In this first section, we investigate the topology of a two-holed cross-surface.  Then we endow it with a hyperbolic structure.

\subsection{Curves and the fundamental group of $\topsurf$}

We begin by reviewing the topology of the two-holed cross-surface
$\topsurf$.  Recall that $X$ and $Y$ are orientation-reversing curves and that $A$ and $B$ bound $\topsurf$. 

Represent $\topsurf$ as the nonconvex component
of the complement of two disjoint discs in $\rpt$ as follows. 
The coordinate axes in $\R^2$ extend to projective lines in $\rpt$
which intersect exactly once (at the origin) in $\R^2$. Choose two disjoint
hyperbolas in $\R^2$ which intersect the ideal line in pairs of points
which do not separate each other. In $\rpt$ these hyperbolas extend to
conics which bound two disjoint convex regions. The complement of these
convex regions is a model for $\surf$.
In  Figure~\ref{fig:hyperbolas}, the complement of $X$ and $Y$  in $\surf$ is 
foliated by curves homotopic to the boundary curves $A$ and $B$.

Choosing orientations and arcs from a basepoint  
$\basepoint$ to each of these four curves,
define elements of $\fg$ corresponding to $A, B, X, Y$. 
Denoting these elements also by $A,B,X,Y$, respectively, 
yields the redundant presentation
\begin{equation}\label{eq:redundant}
\fg \;=\; \langle A, B, X, Y \mid X Y = A, \, \bar{Y} X = B \rangle
\end{equation}
Note that any two-element subset of $\{A,B,X,Y\}$ 
{\em except\/} $\{A,B\}$ freely generates $\fg$
and 
\begin{equation}
A B = X^2.
\end{equation}
Simple closed curves on $\surf$ are easily classified. 
(See for example 
\cite{Bonahon}.)
If $\gamma\subset S$ is a simple closed curve, then exactly one of the following holds:
\begin{itemize}
\item $\gamma$ is contractible and  bounds a disc;
\item $\gamma$ is peripheral and is isotopic to either $A$ or $B$;
\item $\gamma$ is orientation-reversing and is isotopic to either $X$ or $Y$;
\item $\gamma$ is essential, nonperipheral, orientation-preserving and is isotopic to either $X^2$ or $Y^2$.
\end{itemize}

\setcounter{figure}{0}
\begin{figure}[b]
\includegraphics[scale=1.0]
{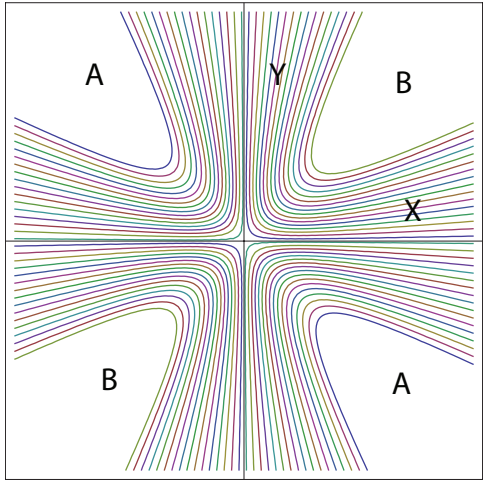}
\caption{The complement of two disjoint discs in 
$\rpt$ is homeomorphic to a $2$-holed cross-surface.
The coordinate axes correspond to the two orientation-reversing 
simple loops $X$ and $Y$.
The two hyperbolas $A$ and $B$ define components of the boundary.}
\label{fig:hyperbolas}
\end{figure}

\subsection{Hyperbolic structures on $\topsurf$}
Consider a complete hyperbolic surface $S$ and a diffeomorphism (the {\em  marking\/})
$\Sigma\longrightarrow S$. 
Choose a universal covering $\tS\to S$, 
and a lift $\tbasepoint\in\tS$ of the basepoint $\basepoint\in\surf$.
Choose a developing map
$\tS\xrightarrow{\dev} \hp$ and holonomy representation 
$$
\fg \xrightarrow{\rho_0} \Gamma \subset \Isom(\hp)\cong\PGLtR.
$$
The holonomy representation embeds $\fg$ as a discrete subgroup
$\Gamma_0 = \rho_0(\fg)$.
The hyperbolic surface $S$ may have finite or infinite area. 

When $\surf$ is nonorientable,
some elements of $\Gamma$  reverse orientation.
When $S$ is a two-holed cross-surface as above,
$\rho_0(X)$ and $\rho_0(Y)$ are glide-reflections, 
while $\rho_0(A)$ and $\rho_0(B)$ are either transvections
or parabolic isometries of $\Ht$.
(Compare Figure~\ref{fig:FundamentalDomain}.)

Henceforth, when the context is unambiguous,
we suppress $\rho_0$, for example, simply writing $A$ for $\rho_0(A)$.

\subsection{Fundamental domains for the two-holed cross-surface}
\label{sec:funddomainsHypSurf}

We now
construct a fundamental domain for the action of the associated group
$\Gamma_0$ on the {\em Nielsen convex region,\/} a convex
$\Gamma_0$-invariant open subset $\Omega \subset\hp$.  Fundamental
domains for incomplete surfaces give rise to fundamental domains
in $\Eto$ for complete Margulis space-times.

Denote the attracting and repelling fixed points of a hyperbolic
isometry $W$ of $\hp$ by $\xp{W}$ and $\xm{W}$ respectively.  
If $W$ is parabolic 
then either annotation represents the fixed point on $\partial\hp$.
Assume the fixed points 
\begin{equation*}
\xp{X},\xp{Y},\xm{X},\xm{Y}
\end{equation*}
are in counter-clockwise order. 
(Compare Fig.\ref{fig:FundamentalDomain}.)
In $\hp$, let $\idealquad$ be the ideal quadrilateral 
with vertices:
\begin{equation*}
\xp{B},\xm{A}, A(\xp{B}), X(\xp{B})
\end{equation*}
(Compare Figure~\ref{fig:FundDomainIdentifiedN1}.)
The transvection $A$ and the glide-reflection $X$ identify the four sides
of $\idealquad$ as follows:
\begin{align*}
\Big(\xm{A},\,\xp{B}\Big)     &\;\stackrel{A}\longmapsto\; \Big(\xm{A},\, A(\xp{B})\Big) \\
\Big(\xp{B},\, X(\xp{B})\Big) &\;\stackrel{X}\longmapsto\; \Big(X(\xp{B}),\, A(\xp{B})\Big)
\end{align*}
since $A(\xp{B}) = X^2(\xp{B})$.
(Compare Figure~\ref{fig:FundDomainIdentifiedN1}.)
The diagonal
$\Big(\xm{A},\, X(\xp{B})\Big)$ divides $\mathcal{Q}$
into two ideal triangles. 
(Compare Figure~\ref{fig:IdealQuadrilateral}.)
Denote the interior of the complement of a halfplane
$\mathfrak{H}$ by $\mathfrak{H}^c$.
The sides of $\idealquad$ bound four disjoint halfplanes which are pairwise
identified by $A$ and $X$ to define a Schottky-like system:
\begin{itemize}
\item
The side $\Big(\xm{A},\, A(\xp{B})\Big)$ bounds a halfplane $\HA$; 
\item
The transvection $\bar{A}$ maps 
$\Big(\xm{A},\, A(\xp{B})\Big)$ 
to the side $\Big(\xm{A},\,\xp{B}\Big)$
bounding the halfplane $\bar{A}(\HA^c)$.
\item The side 
$\Big(X(\xp{B}),\, A(\xp{B})\Big)$ bounds a halfplane $\HX$;
\item
The glide-reflection $\bar{X}$ maps 
$\Big(X(\xp{B}),\, A(\xp{B})\Big)$ to the side
$\Big(\xp{B},\, X(\xp{B})\Big)$ bounding the halfplane $\bar{X}(\HA^c)$.
\item
The four halfplanes 
$$
\HA, \bar{A}(\HA^c), \HX, \bar{X}(\HA^c)
$$ 
are pairwise disjoint.
\end{itemize}
The fundamental domain and its corresponding set of identifications for the
hyperbolic surface $S = \Omega/\Gamma_0$ 
serve as a template for the crooked fundamental domains constructed in
\S\ref{sec:CrookedFD}. The positions of the halfplanes and their bounding
geodesics  play a crucial role in ensuring disjointness of the crooked 
planes.
\newpage

\begin{figure}[t]
\psfrag{m}{$X^-$}
\psfrag{V}{$X^+$}
\psfrag{Z}{$Y^-$}
\psfrag{W}{$Y^+$}
\psfrag{X}{$X$}
\psfrag{Y}{$Y$}
\psfrag{A}{$A$}
\psfrag{B}{$B$}
\psfrag{P}{$A^Y\,=\,A^{\bar{X}}$}
\psfrag{Q}{\qquad\qquad$B^Y\,=\,B^X$}
\psfrag{o}{$\dev(\tbasepoint)$}
\includegraphics[scale=1.1]
{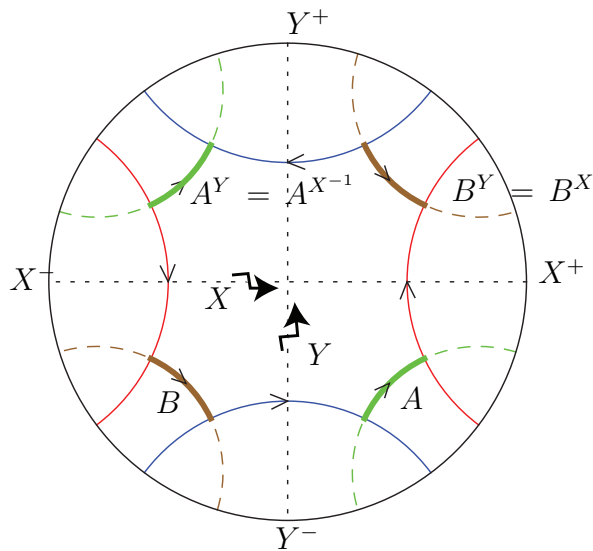}
\caption{
A fundamental domain for the two-holed cross-surface $\surf$.
}
\label{fig:FundamentalDomain}
\end{figure}

\begin{figure}[h]
\psfrag{X}{$X$}
\psfrag{A}{$A$}
\psfrag{B}{$\xp{B}$}
\psfrag{C}{$\xm{A}$}
\psfrag{D}{$A(\xp{B})$}
\psfrag{E}{$X(\xp{B})$}
\includegraphics[scale=1.]
{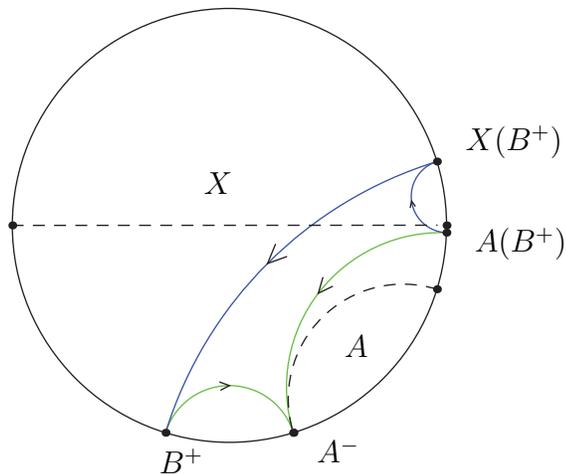} 
\caption{Identifications of the ideal quadrilateral.
The transvection $A$ identifies the bottom two sides,
which share the ideal vertex $\xm{A}$.
The glide-reflection $X$ identifies the top two
sides, which share the ideal vertex $X(\xp{B})$.
The arrows on the sides indicate how to identify the
sides.
}
\label{fig:FundDomainIdentifiedN1}
\end{figure}

\begin{figure}[h]
\psfrag{A}{$\HXM$}
\psfrag{B}{$\xp{B}$}
\psfrag{C}{$\HAM$}
\psfrag{E}{$\xm{A}$}
\psfrag{F}{$\HA$}
\psfrag{G}{$A(\xp{B})$}
\psfrag{H}{$\HX$}
\psfrag{I}{$X(\xp{B})$}
\includegraphics[scale=1.11]
{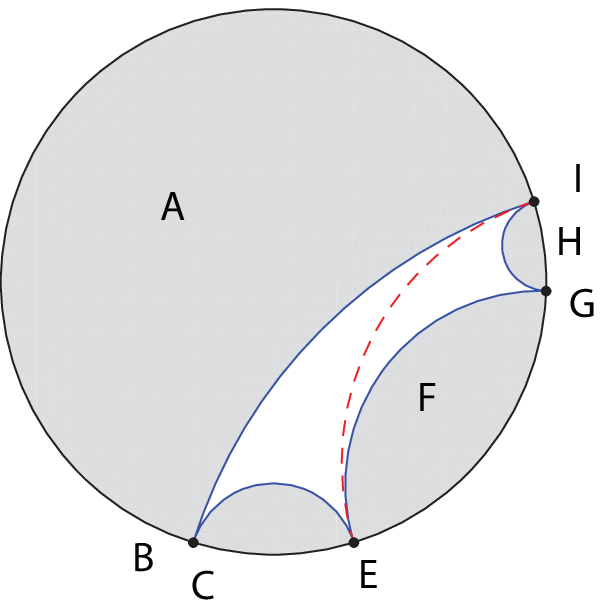}
\caption{The ideal quadrilateral $\idealquad$ with diagonal,
and the four halfplanes $\idealquad$ bounds. 
}
\label{fig:IdealQuadrilateral}
\end{figure}



\section{The Fricke space of $\topsurf$}
This section describes the deformation space of
marked hyperbolic structures on $\topsurf$ in
terms of representations and their characters in $\SLtC$.
Then we find a simple set of coordinates for $\Fricke$
in terms of traces of $2\times 2$ real matrices.

\subsection{Representing orientation-reversing isometries by matrices}
For the following discussion of the Fricke space of the two-holed
cross-surface, represent $\hp$ as a totally geodesic hypersurface in
hyperbolic $3$-space $\hthree$.
{\em Every\/} isometry of $\hp$ extends uniquely to an 
{\em orientation-preserving\/} isometry of $\hthree$ 
preserving $\hp\subset\hthree.$
(Compare \cite{Goldman,Marden}.
To use the algebraic machinery of traces in $\SLtC$, 
we must also lift $\rho_0$ to a representation $\trho_0$ in the double
covering $\SLtC$ of the identity component 
$$
\Isom^+(\hthree)\cong\PSLtC.
$$
Such a lift is always possible since $\fg$ is a free group.
Denote the lifted elements of $\SLtC$ by $A, B, X, Y$ respectively, 
as well.

Following  \cite{Goldman}, represent a lift of an isometry of $\hp$ as an element
of $\SLtC$ which is real if the isometry preserves orientation on $\hp$, 
and purely imaginary it the isometry reverses orientation $\hp$. 
In the real case, the matrix lies in $\SLtR$, and is defined up to $\pm 1$.
It may be elliptic, parabolic or hyperbolic. If it is parabolic, then it is conjugate
to:
\begin{equation}\label{eq:Unipotent}
\pm \bmatrix 1 & 1 \\ 0 & 1 \endbmatrix
\end{equation}
If it is hyperbolic then it is conjugate to the diagonal matrix:
\begin{equation}\label{eq:Transvection}
\pm \bmatrix e^{\ell/2}  & 0 \\ 0 & e^{-\ell/2} \endbmatrix
\end{equation}
with trace $\pm 2 \cosh(\ell/2)$. 
It corresponds to a {\em transvection\/} of $\Ht$.
It leaves invariant a geodesic (its {\em invariant axis),\/}
which for the above example corresponds to the imaginary axis in the upper halfplane model, and it displaces points on its axis by length $\ell$.

In the purely imaginary case, 
the corresponding purely imaginary element of $\SLtC$ is $i P$, 
where $P\in\GLtR$ and $\det(P)\,= \, -1$.
For example the diagonal matrix
\begin{equation}\label{eq:GlideReflection}
iP = i \bmatrix e^{\ell/2 } & 0 \\ 0 & -e^{-\ell/2} \endbmatrix 
\;\in\; \SLtC
\end{equation}
represents a glide-reflection of displacement length $\ell$ along the 
geodesic in the
upper halfplane represented by the imaginary axis. 
Its trace equals $2i\sinh(\ell/2)$. 
Since a matrix in $\SLtC$ representing an isometry of hyperbolic space 
is only determined up to multiplication by $\pm 1$, 
the matrix in $\SLtC$ representing a glide-reflection of
displacement length $\ell$ has trace $\pm 2i \sinh(\ell/2)$.

\subsection{Trace coordinates}
Suppose that $\trho_0$ is a representation in $\SLtC$ 
(preserving $\hyp\subset\hthree$) which covers
a holonomy representation of a marked hyperbolic structure
on $\Sigma$. 

The character of $\trho_0$ corresponds to a quadruple
$(a,b,x,y)\in\R^4$ defined by:
\begin{align*}
a &:= \tr(A) \\
b &:= \tr(B) \\
x &:= -i\ \tr(X)  \\
y &:= -i\  \tr(Y) 
\end{align*}
subject to the trace identity
\begin{equation}\label{eq:trace_identity}
a + b + x y \;=\; 0,
\end{equation}
which arises directly from the ``Basic Trace Identity'' in \cite{Goldman}:
\begin{equation*}
 \tr (XY) + \tr (X\bar{Y}) = \tr (X) \tr (Y)
\end{equation*}
\big(which in turn is just the Cayley-Hamilton theorem for $\SLtC$\big).

The {\em character set\/}
\begin{equation}\label{eq:CharacterSet}
\mathscr{C} :=
\{ (a,b,x,y)\in\R^4 \mid   a + b = - x y,\ \vert a\vert \ge 2, \vert b\vert \ge 2 \}
\end{equation}
identifies with conjugacy classes of lifts to $\SLtC$ of holonomy representations
of marked complete hyperbolic structures on $\Sigma$.

Since $\{X,Y\}$ freely generates $\fg$, different lifts $\trho_0$ of $\rho_0$
differ by {\em multiplication by a character\/}
$$
\fg\xrightarrow{\chi}\{\pm 1\}
$$
as follows. 
The group 
$$
\mathscr{G}\ :=\  \Hom\big(\fg,\{\pm 1\}\big)\ \cong\  \Z/2 \oplus \Z/2
$$ 
acts on $\Hom(\fg,\SLtC)$ by pointwise multiplication:
$$
\gamma \stackrel{\chi\cdot\rho}\longmapsto \chi(\gamma)\rho(\gamma)
$$
which is a representation since $\{\pm 1\} \subset\SLtC$ is central.

On the quotient of $\Hom\big(\fg,\SLtC\big)$ 
by $\Inn\big(\SLtC\big)$, 
the induced $\mathscr{G}$-action
is described by the action on traces: 
\begin{align*}
(a,b,x,y) &\mapsto (a,b,-x,-y); \\
(a,b,x,y) &\mapsto (-a,-b,-x,y); \\
(a,b,x,y) &\mapsto (-a,-b,x,-y).
\end{align*}
Changing the sign of either $\tr(X)$ or $\tr(Y)$ results in changing the sign of $\tr(A)$ and $\tr(B)$, 
while changing the signs of both $\tr(X)$ and $\tr(Y)$ 
leaves the signs of $\tr(A)$ and $\tr(B)$ unchanged. 

The Fricke space $\Fricke$ then identifies with the quotient 
$\mathscr{C}/\mathscr{G}$.
Since none of  $a,b,x,y$ vanish, the action of $\mathscr{G}$ is free and $\Fricke$ also 
identifies with one of the four connected components of the character set.

\subsection{Relating traces to lengths}
The nonperipheral essential orien\-ta\-tion-reversing simple curves $X$ and $Y$ are uniquely
represented by closed geodesics. Similarly the boundary curves $A$ and $B$ are closed
geodesics or cusps. Denote the lengths of these geodesics by $\lX, \lY,\lA, \lB$
respectively. Denote the angle of intersection of the geodesic representatives $X$ and $Y$
by $\theta$. Explicit matrix representatives are given below:

\begin{align*}
X& \longleftrightarrow 
i\bmatrix 
\sinh\frac{\lX}{2} + \cos\theta \cosh\frac{\lX }{2}  &
\sin \theta  \cosh\frac{\lX}{2}  \\
\sin\theta \cosh\frac{\lX}{2} & 
\sinh\frac{\lX}{2} - \cos\theta \cosh\frac{\lX}{2} 
\endbmatrix \\
Y & \longleftrightarrow i\bmatrix  e^{\lY/2 } & 0 \\ 0 & -e^{-\lY/2 }
\endbmatrix.
\end{align*}
Therefore,
\begin{align}\label{eq:trace2length}
x & =\; 2 \sinh\frac{\lX}{2} \\
y & =\;  2 \sinh\frac{\lY}{2} \notag \\
a & =\;  
-2 \big(\sinh\frac{\lX}{2} \sinh\frac{\lY}{2} 
+ \cos\theta  
\cosh\frac{\lX}{2} \cosh\frac{\lY}{2}\big)\notag \\
b& =\; 
-2 \big(\sinh\frac{\lX}{2} \sinh\frac{\lY}{2}
- \cos\theta
\cosh\frac{\lX}{2} \cosh\frac{\lY}{2} \big). \notag 
\end{align}
The defining inequalities for these matrices, $x, y > 0$ and $a, b \le -2$,
describe one component of the set of characters.
The \emph{length identity }
\begin{equation}\label{eq:LengthIdentity}
\cosh \frac{\lA}{2} \,+\, \cosh\frac{\lB}{2}  \;=\;  
2\ \sinh\frac{\lX}{2}  \,  \sinh\frac{\lY}{2} .
\end{equation}
results directly from applying 
\eqref{eq:trace2length} to \eqref{eq:trace_identity}.

\begin{thm}\label{thm:2HXSFrickeSpace}
The lengths  of boundary geodesics $\lA, \lB \geq 0$, and the lengths of 
orientation-reversing simple closed geodesics  $\lX,\lY >0$,
subject to \eqref{eq:LengthIdentity}, 
provide coordinates for $\Fricke$. 
\end{thm}
\noindent
Solve \eqref{eq:LengthIdentity} for $\lY$ in terms of the other
lengths. Thus, $\Fricke$  identifies with the fibered space whose base
is the closed first quadrant in $\R^2$ ($\lA, \lB \geq 0$) and whose
fiber $\R^+$ ($\lX >0$) has no boundary.  That is,
$$
\Fricke \ \approx \ 
\big( [0,\infinity)\times [0, \infinity) \big) \times \R_+ .
$$

\section{Flat Lorentz $3$-manifolds}\label{sec:3manifolds} 

Now we turn from hyperbolic structures on surfaces
to flat Lorentz structures on $3$-manifolds.
After a brief review of Lorentzian geometry, 
we prove that every Margulis spacetime is orientable,
even when the corresponding hyperbolic surface is nonorientable.

\subsection{Linear Lorentzian geometry}
Let $\rto$ denote a Lorentzian $3$-dimensional vector space.
Denote its group of orientation-preserving linear isometries 
$\SOTO$.  In keeping with the notation adopted for $\Isom(\hp)$, 
we use uppercase letters to denote elements of this group.

The nonzero vectors $\vv\in\rto$ such that 
\begin{itemize}
 \item $\vv\cdot\vv =0$ are called \emph{lightlike},
 \item $\vv\cdot\vv <0$ are called \emph{timelike}, and
 \item $\vv\cdot\vv >0$ are called \emph{spacelike}.
\end{itemize}
Any lightlike vector lies on one nappe of the \emph{lightcone}. 
Choose one of these nappes as the  \emph{future lightcone}.
Any vector lying on (lightlike)  or inside (timelike)  the future lightcone is said 
to be \emph{future-pointing}.

The correspondence between the hyperbolic plane and the set of
timelike lines in $\rto$ induces an identification
$\Isom(\hp)\cong\SOTO$.  
Transvections 
correspond to {\em hyperbolic} matrices in $\soto$ with trace greater
than 3. The diagonal matrix \eqref{eq:Transvection} corresponds to 
$$
\bmatrix 1 & 0 & 0 \\ 0 &  \cosh(\ell) & \sinh(\ell) \\
0 & \sinh(\ell) & \cosh(\ell)  \endbmatrix
$$
with trace $1 + 2\cosh(\ell)$.
Parabolic elements of $\SLtR$ correspond to {\em parabolic\/} matrices in 
$\soto$ with trace equal to 3.  The matrix $iP$ of \eqref{eq:GlideReflection}
representing a glide-reflection identifies with the diagonal matrix
$$
\bmatrix 1 & 0 & 0 \\ 0 &  -\cosh(\ell) & -\sinh(\ell) \\
0 & -\sinh(\ell) & -\cosh(\ell)  \endbmatrix
$$
having trace $1 -2\cosh(\ell) < -1$. 
The corresponding isometry of $\Eto$ preserves orientation but reverses time-orientation.

Call $A\in\SOTO$ {\em non-elliptic} if it does not fix a timelike
line.  Equivalently, $A$ corresponds to a glide-reflection, a
transvection or a translation of $\hp$ as above.

Suppose that $A$ is non-elliptic.  Then the 1-eigenspace for $A$ is a
line spanned by a spacelike or a lightlike vector.  We choose a
specific {\em neutral vector}, denoted $A^0$, by requiring that:
\begin{itemize}
\item given a timelike vector $\vv$, $\big( \vv, A^2(\vv), A^0 \big)$ is a right-handed basis for $\rto$.  
(This determines a unique direction for $A^0$.
Furthermore this condition is independent of $\vv$.)
\item if $\Fix(A)$ is spacelike, we choose $A^0$ such that $A^0\cdot A^0=1$.
\end{itemize}
When $A$ is hyperbolic (a transvection or glide-reflection), 
its  eigenvalues are $1$, $\lambda$, and $\lambda^{-1}$, 
where 
$$
0< \lambda^2 <1.
$$
In the above example, $\lambda = \pm e^{-\ell}$.
Choose $A^-$ to be a future-pointing lightlike \emph{contracting} eigenvector, 
and $A^+$ to be a future-pointing \emph{expanding} eigenvector: 
$$
A^2(A^-) = \lambda^2 A^-,\qquad  A^2(A^+) = \lambda^{-2} A^+.
$$  
Both contracting and expanding eigenvectors are lightlike, and the
neutral vector  $A^0$ is spacelike and well-defined.   
(See Figure~\ref{fig:Frames} and compare 
\cite{CharetteDrumm}.) 
\begin{figure}[b]
\includegraphics[width=8cm]{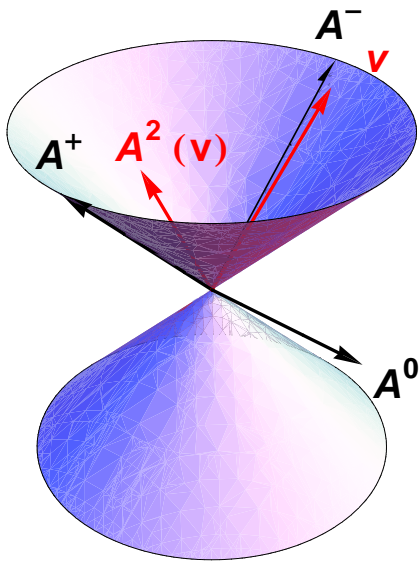}
\caption{Defining the direction of $A^0$.}
\label{fig:Frames}
\end{figure}

When $A$ is parabolic, $A^0$ may be either future-pointing or
past-pointing and its Euclidean length is arbitrary.  By convention,
we set $A^+=A^-$ to be a future-pointing vector that is parallel to
$A^0$.

\subsection{Affine actions}
Let $\Eto$ denote the affine space modeled on $\rto$.  The group of {\em Lorentzian isometries} of $\Eto$, denoted $\Isom(\Eto)$, consists of affine transformations whose linear part preserve the Lorentzian structure on $\rto$.  Equivalently, $\g\in\Isom(\Eto)$ if and only if its linear part $\LL(\g)\in\oto$.

Let $A\in\SOTO$ be non-elliptic.  
Since $A$ is linear and therefore fixes the origin, 
the cyclic group $\langle A\rangle$ it generates does not
act properly on $\Eto$.
However, for any vector $\vu\in\rto$ with nonzero projection on $A^0$,
the following affine transformation acts properly and freely on
$\Eto$:
$$
g: p \longmapsto o+A(p-o) + \vu
$$
where $o$ is some choice of origin.  The quotient $\Eto/\langle g\rangle$ is an open solid torus. 

More generally, if $\Gamma_0\subset\SOTO$ is a free group, we obtain an affine action with linear part $\Gamma_0$ by assigning translation parts to a set of free generators.

\subsection{Margulis spacetimes}
Recall that a Margulis spacetime is a quotient $M = \Eto/\Gamma$
whose fundamental group $\Gamma$ is free, and acts properly, freely and discretely
by Lorentz isometries.

\begin{lemma}\label{lem:orientable}
Every Margulis spacetime with a nonabelian fundamental group is orientable.
\end{lemma}
\begin{proof}
For any Margulis spacetime $M$, 
the affine holonomy group $\Gamma$
acts freely on $\Eto$.
An affine transformation whose linear
part does not have $1$ as an eigenvalue
fixes a point in $\Eto$.
Therefore, for every $\gamma\in\Gamma$,
its linear part 
$\LL(\gamma)$ must have $1$ as an eigenvalue.

If $\LL(\gamma)\in\oto\setminus\SOTO$, then one of its eigenvalues 
equals $-1$.   If $1$ is also an eigenvalue, then $A^2 = I$.
Thus, if $\LL(\Gamma)\not\subset \SOTO$, 
then some $\gamma\in\Gamma$ will have 
trivial linear part, and therefore is
a translation. 
By \cite{FriedGoldman}, $\Gamma$ contains
no translations.
\end{proof}
\section{Affine deformations and cocycles}\label{sec:deformations} 
Affine deformations of a two-holed cross-surface $\surf$
are parametrized by Margulis invariants of $X,Y,A,B$.  
We start with a general
discussion of the space of affine deformations of the linear holonomy
of a surface and the Margulis invariant.  At the end of the section,
we use the Margulis invariant to provide coordinates for the
deformation space for the two-holed cross-surface.

\subsection{Cocycles}
Let $\surf$ be an arbitrary hyperbolic surface with linear holonomy $\Gamma_0=\rho_0(\pi_1(\surf))\subset\SOTO$.  
Recall 
that an affine deformation of $\Gamma_0$
is a lift $\Gamma_0\xrightarrow{\rho}\Isom(\Eto)$ such that $\rho_0 = \LL\circ \rho$.
Lifts correspond to cocycles, that is, maps
$$
\Gamma_0  \stackrel{\vu}\longrightarrow \rto
$$
satisfying
$$
\vu( XY) \;=\; \vu(X) + \rho_0(X) \vu(Y)
$$
Here the affine deformation is defined by:
$$
p \stackrel{\rho(X)}\longmapsto o+\rho_0(X)(p-o) + \vu(X).
$$
Denote the space of cocycles  $\Gamma_0\rightarrow\rto$ by $\ZZ$.

When $\Gamma_0$ %
is free, as is the case for a two-holed cross-surface, a cocycle is determined by its values on a free basis.
Furthermore, the values on a free basis are completely arbitrary.

Two cocycles determine translationally conjugate affine deformations
if and only if they are {\em cohomologous,\/} that is, they differ
by a coboundary
$$
\vu(X) := \vv - \rho_0(X) \vv
$$
where $\vv\in\rto$ is the vector effecting the translation.
The resulting set of translational conjugacy classes of affine deformations
(cohomology classes of cocycles)
compose the cohomology group $\HH$. 

In the present case, when $\Gamma_0$ is free of rank two,
and $\dim(\rto)=3$, the space 
$\ZZ$ of cocycles is $6$-dimensional.
The space of coboundaries is $3$-dimensional.
Therefore the cohomology $\HH$ has dimension $3$.
 
\subsection{The Margulis invariant}

Continuing with the notation above, the {\em Margulis invariant\/} of the affine deformation $\rho(X)$ is the neutral projection of the translational part $\vu$:
$$
\oa{\vu}{X} \;:=\; \vu \cdot X^0
$$
The Margulis invariant 
is everywhere nonzero if and only if the affine deformation
is free.  Moreover if $\vu,\vv$ are cohomologous then $\oa{\vu}{X}=\oa{\vv}{X}$.  Furthermore, fixing the cocycle $\vu$, the Margulis invariant is a class function on $\Gamma_0$.  

The following basic fact
(Margulis's Opposite Sign Lemma) 
was proved in \cite{Margulis1,Margulis2} 
for hyperbolic elements and extended in
\cite{CharetteDrumm} to groups with parabolic elements. 
The survey article~\cite{Abels} containsa lucid description of the ideas in 
Margulis's original  proof. 
\begin{lemma}\label{lemma:oppsign}
Let $g, h\in\Isom(\Eto)$ be non-elliptic.  
If the Margulis invariants for $g$ and $h$ have 
opposite signs then $\langle g , h \rangle$ does not act properly. 
\end{lemma}

\subsection{Deformations of hyperbolic structures}

Affine deformations also correspond to 
{\em infinitesimal deformations\/}
of the marked hyperbolic surface $\Sigma \to S$
(
\cite{Goldman_MargulisInvariant,GoldmanMargulis}).
Consider a smooth family $S_t$ of hyperbolic surfaces, with 
smoothly varying markings 
$$
\Sigma \xrightarrow{m_t} S_t
$$ 
and holonomy representation 
$\pi_1(\Sigma)\xrightarrow{\rho_t}\SLtC$.  
Let $\vu$ be the cocycle tangent to the space of representations
at $t=0$ .
By \cite{GoldmanMargulis},
the Margulis invariant identifies with the derivative of the {\em geodesic length function $\ell_\gamma$,\/}
where $\gamma\in\fg$. Specifically, let $\ell_{\gamma(t)}\in\R_+$ denote the length of
the closed geodesic in $S_t$ in the free homotopy class determined by $(m_t)_*(\gamma)\in\pi_1(S)$.
Then:
$$
\oa{\vu}{\rho_0(\gamma)} \;=\;  \frac{d \ell_{\gamma(t)} }{dt} .
$$
Furthermore: 
\begin{equation}\label{eq:Lengths2Traces}
\tr\big( \rho_t(\gamma)\big) \;=\; \pm 2 
\begin{cases}
\cosh\frac{\ell_{\gamma(t)}}{2} &\text{~if~} \gamma 
\text{~preserves orientation} \\
i \sinh\frac{\ell_{\gamma(t)}}{2} &\text{~if~} \gamma 
\text{~reverses orientation} 
\end{cases}
\end{equation}
Differentiating  \eqref{eq:Lengths2Traces} implies:
\begin{equation}\label{eq:dLengths2Traces}
\frac{d}{dt} \tr \big( \rho_t(\gamma) \big)   \;=\; \pm
\begin{cases}
\alpha_\gamma \sinh\frac{\ell_{\gamma(t)}}{2} &\text{~if~} \gamma 
\text{~preserves orientation} \\
i\alpha_\gamma \cosh\frac{\ell_{\gamma(t)}}{2} &\text{~if~} \gamma 
\text{~reverses orientation} 
\end{cases}
\end{equation}

\subsection{Coordinates for the four-sided deformation space}
\label{sec:coordinates}

Fixing $A\in\Gamma_0$, the Margulis invariant is a
linear functional on the space of cocycles $\ZZ$, descending to a
well-defined functional on $\HH$.  To reflect this, we modify our
notation for the Margulis invariant as in~\cite{CDG}:
\begin{align*}
\HH & \xrightarrow{\na_A} \R \\
[\vu] & \longmapsto \oa{\vu}{A}
\end{align*}

Now assume $\surf$ is a two-holed cross-surface, with elements $A,B,X,Y\in\pi_1(\surf)$ as in \S\ref{sec:thx}.  The invariants of elements $A,X,\XM A=Y$ determine an isomorphism of vector spaces:  
\begin{align}\label{eq:CDM-coordinates}
\HH &\xrightarrow{\na} \mathbb{R}^3  \\
 [u]  & \longmapsto
\bmatrix \na_A(u) \\ \na_X(u) \\\na_Y(u) 
\endbmatrix  \notag
\end{align}

Consider a proper affine deformation.  Lemma~\ref{lemma:oppsign}
implies the Margulis invariants of all nontrivial elements have the
same sign.  Namely, consider an infinitesimal deformation of the
hyperbolic structure on $\surf$ such that every closed geodesic 
{\em infinitesimally lengthens} or an infinitesimal deformation where
every closed geodesic \emph{infinitesimally shortens}. 

Differentiating \eqref{eq:trace_identity} yields:
\begin{equation}\label{eq:deriv_trace_identity}
da \, +\, db \,+\, y\ dx \ \,+\, \ x\ dy \;=\; 0.
\end{equation}
Express these quantities using \eqref{eq:trace2length}, 
its derivatives
\eqref{eq:dLengths2Traces},  and use the fact that  
the Margulis invariants $\na_A,\na_B,\na_X,\na_Y$
of $A,B,X,Y$, are
the derivatives $d\lA, d\lB, d\lX, d\lY$  
respectively \cite{GoldmanMargulis,CharetteDrumm}.
The differentiated trace identity 
\eqref{eq:deriv_trace_identity} then implies:


\begin{align}\label{eq:alphaCDMrelations}
&\Big(\sinh\frac{\lA}{2}\Big)   \na_A +  \Big(\sinh\frac{\lB}{2}\Big) \na_B = \\
 &  \qquad\qquad  \Big(2\cosh\frac{\lX}{2} \sinh\frac{\lY}{2}\Big) \na_X \ +\ 
\Big(2 \sinh\frac{\lX}{2} \cosh\frac{\lY}{2}\  \Big) 
\na_Y. \notag  
\end{align}
After eliminating $\na_B$ using \eqref{eq:alphaCDMrelations}, 
the positivity of $\na_A$ and
$\na_B$ is written as follows:
\begin{align}\label{eq:muAmuXmuY}
0 \   & <  \na_A  \\ 
&  < 
  \bigg( 2\csch\frac{\lA}{2} \cosh\frac{\lX}{2} \sinh\frac{\lY}{2} \bigg)
\ \na_X  \notag \\ 
& \qquad + 
\bigg( 2\csch\frac{\lA}{2} \sinh\frac{\lX}{2} \cosh\frac{\lY}{2}\bigg)\ 
\na_Y \notag
\end{align}
Recall that $\Dd$ denotes the set of deformations where the
functionals $\na_X , \na_Y , \na_A , \na_B $ all have the same sign.  Let $\Dd_+\subset\Dd$ denote those whose signs are all positive.  Then the set of lengthening deformations lies inside the
set $\Dd_+$
and is defined in terms of the linear coordinates $\na_A,\na_X,\na_Y$
by \eqref{eq:muAmuXmuY} and the conditions $\na_X > 0$ and $\na_Y > 0$.

$\Dd_+$ is a four-sided cone, and is thus spanned by four rays.
Each of these rays can be defined by
the linear equation \eqref{eq:alphaCDMrelations} together
with one of the four conditions:
\begin{align*}
\na_Y, \na_B > 0,\; &\na_X = \na_A = 0 \\
\na_X, \na_B > 0,\; &\na_Y = \na_A = 0 \\
\na_X, \na_A > 0,\; &\na_Y = \na_B = 0 \\
\na_Y, \na_A > 0,\; &\na_X = \na_B = 0.
\end{align*}
\noindent 
The set of shortening deformations lies inside $\Dd_-$, where 
$$
\Dd_- = -\Dd_+. 
$$ 

\begin{figure}[bh]
\includegraphics[scale=2.3]{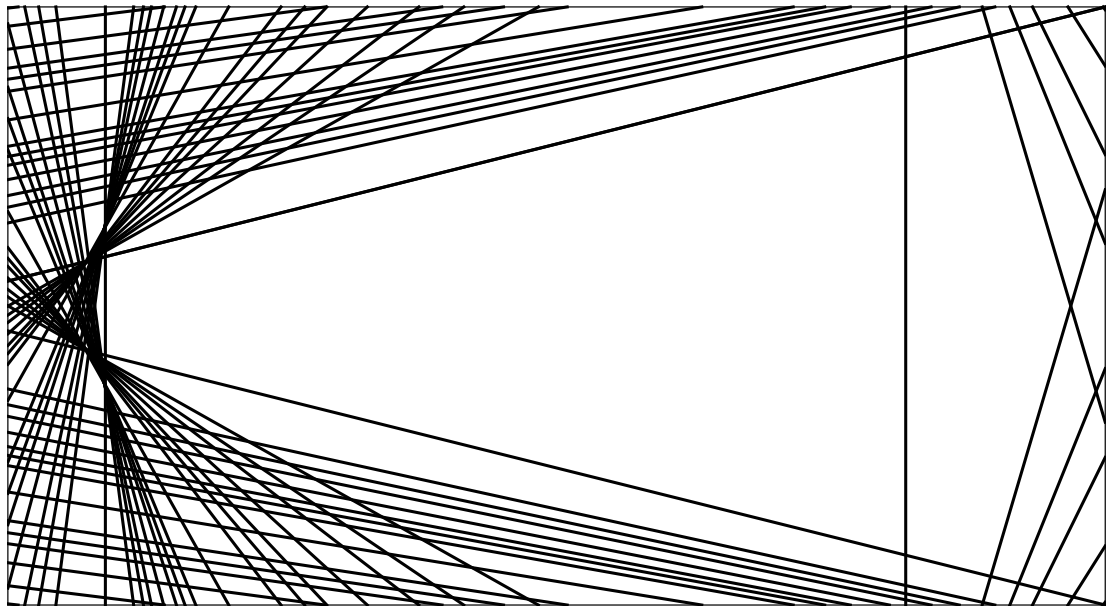}
\caption{The four-sided deformation space
for a two-holed cross-surface}
\label{fig:DeformationSpace}
\end{figure}

\section{Crooked fundamental domains}

In this final section we prove the Main Theorem stated in the
Introduction.  Explicitly, we show that every proper affine
deformation admits a crooked fundamental domain bounded by four
disjoint crooked planes.  It follows that $\Eto/\Gamma$ is a solid
handlebody of genus two.

As in the case of the three-holed sphere~\cite{CDG}, 
it suffices
to consider all affine deformations arising from a single
configuration of crooked planes. This configuration is modeled
on the fundamental domain for $S$ developed in 
\S\ref{sec:funddomainsHypSurf}. 
These crooked plane configurations correspond to 
decompositions of $S$ into two ideal triangles  bounded by three different geodesics. Some of these crooked
plane configurations also describe the entire space of proper affine deformations. 
The more complicated
case of the one-holed torus requires the use of multiple configurations
of crooked planes to cover all the proper affine deformations~\cite{CDGp}.
That is why we obtain a finite-sided
deformation space, in contrast with the case of the one-holed torus,
for example.

We provide definitions and some background on crooked planes in the Appendix.

\subsection{Configuring crooked planes}\label{sec:CrookedFD}

Given a spacelike vector $\vv\in\rto$ and $p\in\Eto$, let $\CP(\vv,p)$
denote the crooked plane with direction vector $\vv$ and vertex $p$.
Its complement in $\Eto$ is a pair of crooked halfspaces.

As one-dimensional spacelike subspaces of $\rto$ bijectively
correspond to hyperbolic lines in $\hyp$, so do parallelism classes of
crooked planes.  Thus to any configuration of lines in $\hyp$ we can
associate a corresponding configuration of crooked planes, up to
choice of vertices.  In particular, every ideal triangle in $\hp$
corresponds to a triple of crooked planes such that every pair of
direction vectors shares a common orthogonal lightlike line.  Moreover
such a triple determines a triple of pairwise disjoint crooked
halfspaces, for appropriate choices of vertices.

Crooked planes enjoy the following useful property:
orientation-preserving Lorentzian isometries map crooked planes to
crooked planes.  Explicitly, if $g\in\SOTO$ then:
$$
g\big(\CP(\vv,p)\big)=\CP\big(\LL(g)(\vv),g(p)\big).$$ 
Moreover, if these crooked
planes are disjoint, then there exists a crooked halfspace $\mathcal{H}$ in the
complement of $\CP(\vv,p)$ such that the closures of $\mathcal{H}$ and $g(\mathcal{H})$
are disjoint.  Thus crooked planes are suitable for Klein-Maskit
combination arguments, {\em even for groups containing  glide-reflections}.

\subsection{Disjointness}

Start with four crooked planes all with the same vertex and
corresponding to the ideal quadrilateral $\idealquad$ pictured in
Figure~\ref{fig:IdealQuadrilateral}.  We  associate vertices to
each crooked plane which make them disjoint.  This assignment of
vertices yields a proper affine deformation of $\Gamma_0=\langle
A,X\rangle$.

We introduce 
a convention 
on 
lightlike vectors.  Recall that glide-reflections in $\Isom(\hp)$
identify with isometries in $\SOTO$ which interchange the future and
past null cones.  To facilitate calculations, take all lightlike vectors
to be future-pointing. If $X$ corresponds to a
glide-reflection and $\vv$ is a future-pointing vector then write all
expressions involving the action of $X$ on $\vv$ in terms of $X(-\vv)$
which is again future-pointing.

As in~\cite{CDG}, it suffices to consider a triple of crooked planes
corresponding to one of the ideal triangles forming $\idealquad$.  Let
$\vv_0$, $\vv_A$, $\vv_X$ be a triple of consistently oriented
unit-spacelike vectors such that:
\begin{align*}
\vv_0^\perp & = \langle\xm{A},X(-\xp{B})\rangle \\
\vv_A^\perp & = \langle A(\xp{B}),\xm{A}\rangle \\
\vv_X^\perp & = \langle X(-\xp{B}),A(\xp{B}) \rangle.
\end{align*}
These will be the directing vectors of our triple of crooked planes.  

Given $(p_0,p_A,p_X)\in \Eto \times \Eto \times \Eto$, 
define the cocycle 
$\vu_{(p_0,p_A,p_X)} \;\in\; \ZZ
$ 
as follows:

\begin{align*}
\vu_{(p_0,p_A,p_X)}(A) &:=  p_A - p_0  \\
\vu_{(p_0,p_A,p_X)}(X) &:=  p_X - p_0.
\end{align*}
Since $A, X$ freely generate $\fg$, 
these conditions uniquely determine the cocycle $\vu_{(p_0,p_A,p_X)}$.

This cocycle corresponds to the following triple of crooked planes: 
\begin{align*}
\CP_0 & := \;
\CP(\vv_0 ,p_0) \\
\CP_A & := \;
\CP(\vv_A , p_A) \\
\CP_X & := \;
\CP(\vv_X,  p_X).
\end{align*}
If $\rho$ is the affine deformation corresponding to
$\vu_{(p_0,p_A,p_X)}$, then $\CP_0$ bounds a crooked halfspace whose
closure contains $\rho(A)^{-1}(\CP_A)$ and $\rho(X)^{-1}(\CP_X)$.
Note that the quadruple
$$ 
\CP_A,\CP_X,\rho(A)^{-1}(\CP_A),\rho(X)^{-1}(\CP_X) 
$$ 
corresponds to $\idealquad$.

We would like these crooked planes to be disjoint.  
To this end, consider the three quadrants:
\begin{align*}
\Q_0 & :=\; \Quad{\xm{A}, -X(-\xp{B})}  \\
\Q_A & :=\; \Quad{A(\xp{B}), -\xm{A}} \\
\Q_X & :=\; \Quad{X(-\xp{B}),- A(\xp{B})} \\
\end{align*}
By Lemma~\ref{lem:disjointasym}, the triple of crooked planes $\CP_0,\CP_A ,\CP_X$ are pairwise disjoint as long as $(p_0,p_A,p_X)\in \Q_0 \times \Q_A \times \Q_X$.

\begin{thm}\label{thm:tame}
Let 
$$
(p_0,p_A,p_X)\in \Q_0 \times \Q_A \times \Q_X.
$$
Then the affine deformation defined by
$$
\vu_{(p_0,p_A,p_X)} \;\in\; \ZZ
$$
is tame.
\end{thm}

\begin{proof}
Apply the Kissing Lemma (Lemma~\ref{lem:kissing1}) to the crooked planes 
in the halfspace bounded by $\CP(\vv_0,p_0)$, to obtain pairwise disjoint crooked planes.  By Theorem~\ref{thm:disjointCPs}, they bound a fundamental domain.
\end{proof}

Figure~\ref{fig:FourwTwoKiss} shows a quadruple of crooked planes, to which we apply the Kissing Lemma to obtain pairwise disjoint crooked planes. 

\begin{figure}
\includegraphics[width=6.3in]{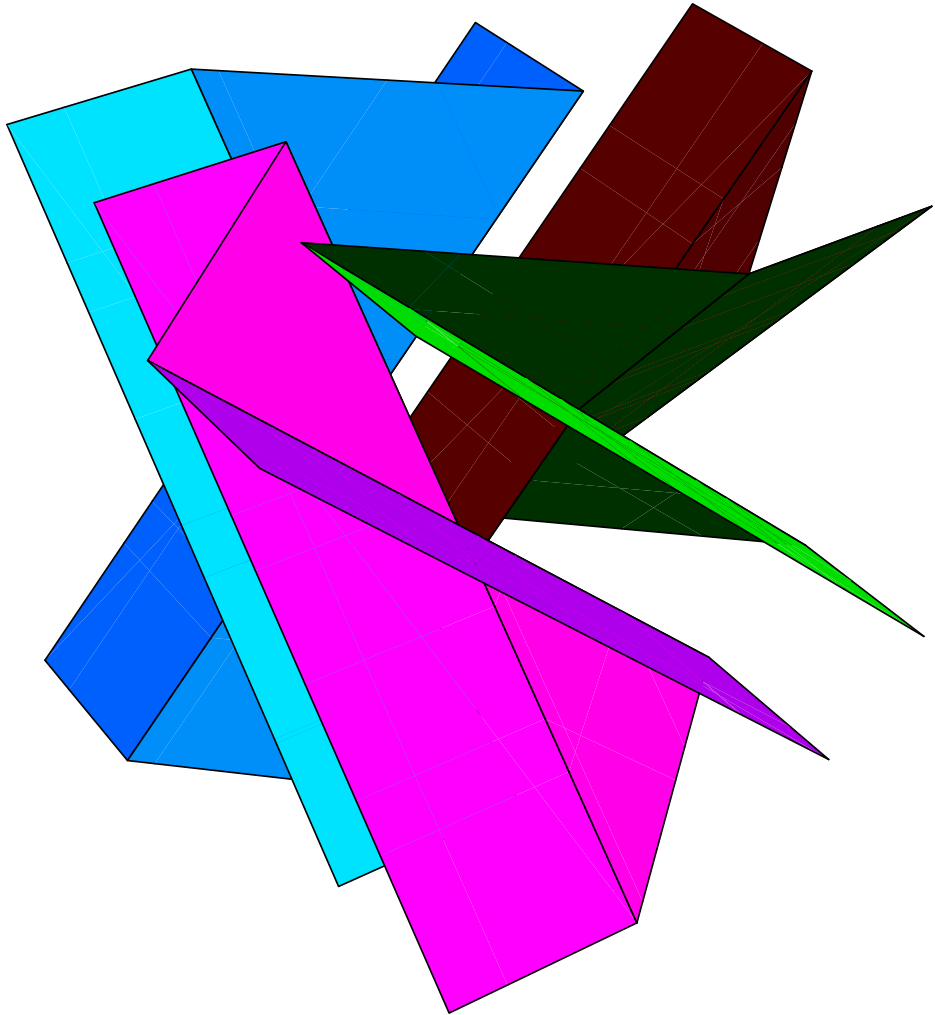} 
\caption{Four crooked planes.  The two on the right share a wing.}
\label{fig:FourwTwoKiss} 
\end{figure}

\subsection{Proper deformations are tame}
Now we conclude the proof.  Recall that all tame deformations are
proper.  The Margulis invariants  of the four
elements $X,Y,A,B$ are all of the same sign for every proper
deformation.  That is,
$$
\Ta \; \subset \; \Pr \; \subset \; \Dd .
$$
We now show that $\Dd \; \subset \; \Ta$ .

As before, Lemma~\ref{lemma:oppsign} implies that each of these 
sets has two connected components which are opposites of each other. 
For proper deformations, 
$$
\Pr= \Pr_+\cup \Pr_-
$$ 
and $\Pr_-= - \Pr_+$. 
For tame deformations, 
$$
\Ta= \Ta_+ \cup \Ta_-
$$ 
and $\Ta_- = - \Ta_+$. 
Thus 
\begin{align*}
\Ta_+   & \subset \Pr_+  \subset \Dd_+ \\
\Ta_-   & \subset \Pr_-  \subset \Dd_- 
\end{align*}
Theorem~\ref{thm:tame} asserts that the image of the mapping
\begin{align*}
\Q_0 \times \Q_A \times \Q_X &\longrightarrow \HH \\
(p_0,p_A,p_X) &\longmapsto [\vu_{(p_0,p_A,p_X)}]
\end{align*}
lies in $\Ta$.

In fact we shall show that this image is exactly
$\Dd_+$. The proof for the positive components 
implies 
the same statements for the negative components.

We will in fact consider cohomology classes in the closure of $\Ta_+$.  Write:
\begin{equation}\label{eq:vertexdef}
\begin{array}{rcl}
p_0 &:= &r_0 \xm{A} - s_0 X(-\xp{B}) \\
p_A &:= &r_A A(\xp{B}) - s_A \xm{A}  \\
p_X &:=& r_X X(-\xp{B}) - s_X A(\xp{B}) 
\end{array}
\end{equation}
where 
$$
r_0, s_0, r_A, s_A, r_X, s_X \geq 0.
$$
Assigning a zero value to any of these coefficients yields a configuration of 
{\em kissing\/} crooked planes, any two of  which intersect in a point, 
a ray or a halfplane.

Applying the definition of a cocycle:
\begin{lemma}\label{lem:yb}
For any cocycle $\vu$,
\begin{itemize}
\item $\vu(Y) \;=\; \bar{X} \big(\vu(A) - \vu(X)\big)$;
\item $\vu(B) \;=\; \bar{Y} \big(\vu(X) - \vu(Y)\big)$.
\end{itemize}\end{lemma}
\noindent 
Consequently:
\begin{equation}\label{eq:udef} 
\begin{array}{rcl}
\vu(A)   & = &  p_A - p_0 \\
\vu(X)	 & = &  p_X - p_0  \\
\vu(Y) 	 & = & \bar{X}(p_A - p_X) \\
\vu(B)	 & = &  \bar{Y}(p_X - p_0)- \bar{A}( p_A - p_X )  
\end{array}
\end{equation}
Therefore:
\begin{equation}\label{eq:margulisoftriple}
\na([\vu_{(p_0,p_A,p_X)}])=
\begin{bmatrix}
(r_A\xp{B}+s_0X(-\xp{B}))\cdot A^0 \\
(-(r_X+s_0)\xp{B}-s_XA(\xp{B})-r_0\xm{A})\cdot X^0\\
((r_A+r_X+s_X)\xp{B}-s_AX^{-1}(\xm{A}))\cdot Y^0
\end{bmatrix}
\end{equation}

\noindent
The following lemma is also immediate:
\begin{lemma}\label{lem:alphasXYA}
Let $(p_0,p_A,p_X)\in \Q_0 \times \Q_A \times \Q_X$ be as above.
\begin{itemize}
\item 
$\vu_{(p_0,0,0)}(Y) \;=\; 0$ so 
$\na_Y(\vu_{(p_0,0,0)}) \;=\; 0$.
\item 
$\vu_{(0,p_A,0)}(X) \;=\; 0$ so 
$\na_X(\vu_{(0,p_A,0)}) \;=\; 0$.
\item 
$\vu_{(0,0,p_X)}(A) \;=\; 0$ so 
$\na_A(\vu_{(0,0,p_X)}) \;=\; 0$.
\end{itemize}
\end{lemma}
%

The deformation space $\Dd_+$ is the (open) convex hull of four rays.
%
We will exhibit a cocycle on each ray and show that its $\mu$-coordinates are all nonnegative. 

Figure~\ref{fig:FuturePointing} serves as a visual aid in the
calculation of $\mu$-coordinates, by showing where the relevant
future-pointing vectors are located relative to the axes of $X$ and $Y$. 

%
\begin{figure}[h]
\psfrag{m}{$X^-$}
\psfrag{V}{$X^+$}
\psfrag{Z}{$Y^-$}
\psfrag{W}{$Y^+$}
\psfrag{X}{$X$}
\psfrag{Y}{$Y$}
\psfrag{A}{$A$}
\psfrag{B}{$B$}
\psfrag{P}{$A^Y\,=\,A^{\bar{X}}$}
\psfrag{Q}{$B^Y\,=\,B^X$}
\psfrag{o}{$\dev(\tbasepoint)$}
\includegraphics[scale=0.6] 
{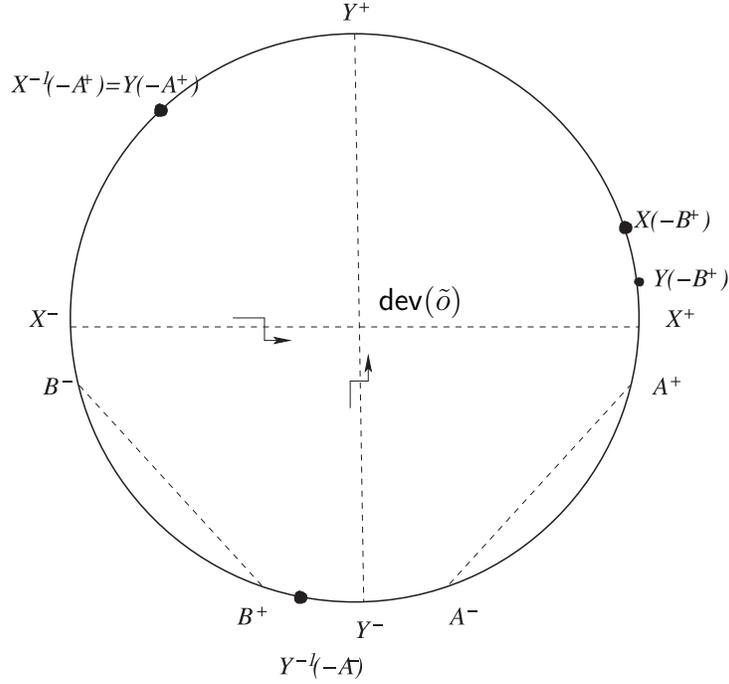}
\caption{Computing cocycles on the edge of $\Dd_+$. The locations of relevant future-pointing vectors are shown.  The figure depicts the case when $A,B$ are hyperbolic; in the case where either isometry is parabolic, the corresponding invariant axis shrinks to a point.}
\label{fig:FuturePointing}
\end{figure}

\bigskip
\noindent{{\bf First ray:~}$\na_B^{-1}(0)\cap\na_Y^{-1}(0)$}

\smallskip
\noindent
For $\na_B =\na_Y  =  0$, 
let $(p_0,p_A,p_X)  =  (-s_0 X(-\xp{B}), 0, 0)$ with $s_0>0$.
\begin{equation*}
\begin{array}{rcl}
\vu(A)   & = &  s_0 X(-\xp{B}) \\
\vu(X)	 & = &  s_0 X(-\xp{B}) \\
\vu(Y) 	 & = & \mathsf{0} \\
\vu(B)	 & = &  \bar{Y}(s_0 X(-\xp{B})) = -s_0\xp{B}
\end{array}
\end{equation*}

\noindent
Because $\xp{B}\cdot \xo{B} = 0$, we know that
$\na_Y(\vu)=\na_B(\vu)=0$. Recall that $X$ is a glide
reflection, so that the future pointing vector $X(-\xp{B})$ lies
above   
the axis of $X$ in
Figure~\ref{fig:FuturePointing}. Therefore, $X(-\xp{B}) \cdot
\xo{X}>0$ and $X(-\xp{B}) \cdot \xo{A}>0 $, so that
$$
\na(A) =s_0 X(-\xp{B}) \cdot  \xo{A} >0
$$
and
$$
\na(X) =s_0 X(-\xp{B}) \cdot  \xo{X} >0.
$$

\bigskip
\noindent{{\bf Second ray:~}$\na_B^{-1}(0)\cap\na_X^{-1}(0)$}

\smallskip

\noindent
For $\na_B = \na_X  =  0$, let 
$(p_0,p_A,p_X)  =  (0, r_A A(\xp{B}), 0)$ with $r_A >0$.
\begin{equation*}
\begin{array}{rcl}
\vu(A)   & = &  r_A A(\xp{B}) \\
\vu(X)	 & = &  \mathsf{0} \\
\vu(Y) 	 & = & \bar{X}(r_A A(\xp{B})) =  - r_A Y(-\xp{B}) \\
\vu(B)	 & = &  -\bar{A}(r_A A(\xp{B})) = -r_A \xp{B}
\end{array}
\end{equation*}
Clearly  $\na_X(\vu)=\na_B(\vu)=0$. The future pointing 
vector $Y(-\xp{B})$ lies to the right of the axis of $Y$ in 
Figure~\ref{fig:FuturePointing},  so that
$Y(-\xp{B}) \cdot  \xo{Y}<0$. Because $\xp{B} \cdot  \xo{A}>0$, then $A(\xp{B}) \cdot  \xo{A}>0$.
Thus, 
$$
\na(A) =r_A A(\xp{B}) \cdot  \xo{A} >0 
$$
and
$$
\na(Y) =- r_A Y(-\xp{B})  \cdot  \xo{Y} >0 .
$$

\bigskip
\noindent{{\bf Third ray:~}$\na_A^{-1}(0)\cap\na_X^{-1}(0)$

\smallskip
\noindent
For $\na_A =\na_X  =  0$, let 
$(p_0,p_A,p_X)  =  (0, -s_A \xp{A}, 0)$ with $s_A>0$.
\begin{equation*}
\begin{array}{rcl}
\vu(A)   & = &  -s_A \xp{A}\\
\vu(X)	 & = &  \mathsf{0} \\
\vu(Y) 	 & = & \bar{X}(-s_A \xp{A})\ =\  s_A \bar{X}(-\xp{A}) \\
\vu(B)	 & = &  -\bar{A}(-s_A \xp{A})\ =\  s_A \xp{A}
\end{array}
\end{equation*} 
It is clear that $\na_A(\vu)=\na_X(\vu)=0$. Because $\xp{A}$ is a fixed point
of $A = XY$, the future pointing vector 
$\bar{X}(-\xp{A})=Y(-\xp{A})$ lies to the left of the axis for $Y$ so that
$\bar{X}(-\xp{A}) \cdot  \xo{Y}>0$. It is also evident that $\xp{A} \cdot  \xo{B}>0$.
Thus
$$
\na(Y) =s_A \bar{X}(-\xp{A})  \cdot  \xo{Y} >0
$$
and
$$
\na(B) =- s_A \xp{A}  \cdot  \xo{B} >0 
$$

\newpage
\noindent{{\bf Fourth ray:~}$\na_A^{-1}(0)\cap\na_Y^{-1}(0)$}

\smallskip

\noindent
For $\na_A\;=\;\na_Y \; = \; 0$, let  
$(p_0,p_A,p_X) \; = \; (r_0 \xm{A}, 0, 0)$ with $r_0>0$.
\begin{equation*}
\begin{array}{rcl}
\vu(A)   & = &  -r_0 \xm{A}\\
\vu(X)	 & = &  -r_0 \xm{A} \\
\vu(Y) 	 & = & \mathsf{0} \\
\vu(B)	 & = &  \bar{Y}(-r_0 \xm{A})\ =\  r_0\bar{Y}(-\xm{A}) .
\end{array}
\end{equation*}
Certainly $\na_A(\vu)=\na_Y(\vu)=0$. It is also clear that 
$\xm{A} \cdot  \xo{X}<0$.  

In order to locate $\bar{Y}(-\xm{A})$, observe that $\xm{A}$ is inside the
circular arc from $X(-\xp{B})$ clockwise to  $\xm{Y}$ 
(see Figure~\ref{fig:IdealQuadrilateral}). 
Because $B=\bar{Y}X$,  $X(-\xp{B}) = Y(-\xp{B})$ and $Y$ maps the arc from
$X(-\xp{B})$ clockwise to  $\xm{Y}$  to the arc from $\xm{Y}$ to $\xp{B}$. 
In particular, $\bar{Y}(-\xm{A}) \cdot  \xo{B}>0$ 
so that: 
$$
\na(X) =-r_0 \xm{A}  \cdot  \xo{X} >0
$$
and
$$
\na(B) =r_0\bar{Y}(-\xm{A}) \cdot  \xo{B} >0 
$$

Therefore, every cohomology class in the interior of the convex hull of these four rays is represented by a cocycle $\vu_{(p_0,p_A,0)}$ such that 
$$
r_0, s_0, r_A, s_A>0.
$$  
This only yields a configuration of crooked planes where $\CP_0$ and $\CP_A$ are disjoint, but each kisses $\CP_X$.  However, 
Equation~\eqref{eq:margulisoftriple} implies that by slightly perturbing $s_A$ and $r_0$, we can make $r_X$ and $s_X$ both strictly positive within the same cohomology class.  Finally, since the new values $(p'_0,p'_A,p_X)$ are in $\Q_0 \times \Q_A \times \Q_X$,  $[\vu_{(p_0,p_A,0)}]\in\Ta_+$ and thus $\Dd_+\subset\Ta_+$ as claimed.
\qed

\appendix
\section{Some facts about crooked planes}

We present here some relevant definitions and facts about crooked planes and the fundamental domains they bound.   We refer the reader to~\cite{Drumm_thesis,DrummGoldman_crooked} for proofs and background, as well as \cite{CDG}, where we first considered triples of crooked planes.

Let $p\in\Eto$ be a point and $\vv\in\rto$ a spacelike vector.
Define the {\em crooked plane}
$\CP(\vv,p)\subset\Eto$ with {\em vertex\/} $p$ and
{\em direction vector\/} $\vv$
to be the union of two {\em wings}
\begin{align*}
& p +\R\xp{\vv}+\R_+\vv^0\\
& p +\R\xm{\vv}-\R_+\vv_0
\end{align*}
and a {\em stem}
\begin{equation*}
p +\  \{\vx\in\rto \mid\ \ldot{\vv}{\vx} = 0,
\ldot{\vx}{\vx} \le 0 \} .
\end{equation*}

\begin{defn}
Let $\vv$ be a spacelike vector and $p\in\Eto$.
The {\em crooked halfspace} determined by $\vv$ and $p$,
denoted $\H(\vv,p)$, consists of all $q\in\Eto$ such that:
\begin{itemize}
\item $\ldot{(q-p)}{\xp{\vv}}\leq0$ if $\ldot{(q-p)}{\vv}\geq 0$;
\item $\ldot{(q-p)}{\xm{\vv}}\geq0$ if $\ldot{(q-p)}{\vv}\leq 0$;
\item  either condition must hold for $q-p\in\vv^\perp$.
\end{itemize}
\end{defn}

Observe that while $\CP(\vv,p)=\CP(-\vv,p)$, the crooked halfspaces $\H(\vv,p)$ and $\H(-\vv,p)$ are distinct components of the complement of $\CP(\vv,p)$.

Crooked planes serve to prove a Klein-Maskit combination theorem for free Lorentzian groups.  The following theorem, proved for subgroups of $\soto$, extends to groups containing glide-reflections.

\begin{thm}\cite{Drumm_thesis}\label{thm:disjointCPs}
Suppose that
$\H(\vv_i,p_i)$ are $2n$ pairwise disjoint crooked halfspaces and $\g_1, \ldots \g_n$ are non-elliptic elements such that for
all $i \in \{\pm 1,\dots,\pm n\}$,
\begin{equation*}
\g_i\big(\H(\vv_{-i},p_{-i})\big) \;=\;
\Eto\setminus \mathsf{int} \big(\H(\vv_{i},p_i)\big).
\end{equation*}
Then
$\G=\langle \g_1, \ldots \g_n \rangle$ acts freely and properly on
$\Eto$ with fundamental domain
\begin{equation*}
\Omega = \Eto \setminus \bigcup_{-n\leq i\leq n,\  i\neq 0}\mathsf{int}\big(
\H(\vv_i,p_i) \big) .
\end{equation*}
\end{thm}

We therefore need a condition for disjointness of crooked planes.  
Start with a technical definition.

 \begin{defn}
\label{def:co} Spacelike vectors $\vv_1, \dots, \vv_n\in\rto$
are
{\em consistently oriented} if and only if, whenever $i\neq j$,
\begin{itemize}
\item $\ldot{\vv_i}{\vv_j}<0$;
\item $\ldot{\vv_i}{\xpm{\vv_j}}\leq 0$.
\end{itemize}
\end{defn}

\begin{lemma}\cite{CDG}
\label{lem:disjointasym}
Let $\vv_1, \vv_2\in\rto$ be non-parallel, consistently oriented 
vectors such that $\xm{\vv_1}=\xp{\vv_2}$. Suppose
\begin{equation*}
p_i=a_i\xm{\vv_i} -b_i \xp{\vv_i},
\end{equation*}
where $a_i,b_i>0$ for $i=1,2$.
Then $\CP(\vv_1,p_1)$ and $\CP(\vv_2,p_2)$ are disjoint.
\end{lemma}

Finally, while Theorem~\ref{thm:disjointCPs} requires pairwise disjoint crooked planes, some of the configurations we encounter contain crooked planes that share a vertex, a line or even a wing.  Nonetheless, these crooked planes can be separated in cases of interest, as described in the following lemma. 

\begin{lemma}[Kissing Lemma, \cite{CDG}]
\label{lem:kissing1}
Let $\vu_1,\vu_2,\vv_1,\vv_2\in\rto$ be  pairwise consistently oriented vectors
and
suppose
$q,p_1,p_2\in\Eto$ satisfy:
\begin{align*}
\CP(\vv_1,p_1)\,\cap\,\CP(\vv_2,p_2) &=\;\emptyset\\
\CP(\vv_{1},p_{1})\,\cap\,\CP(\vu_1,q) &=\;\CP(\vv_{1},p_{1})\,\cap\,\CP(\vu_2,q) =\;\emptyset \\
\CP(\vv_{2},p_{2})\,\cap\,\CP(\vu_1,q) &=\;\CP(\vv_{1},p_{1})\,\cap\,\CP(\vu_2,q) =\;\emptyset .
\end{align*}
Let $\g_1,~\g_2\in\Isom(\Eto)$ such that $\g_i(\mathcal{H}(\vu_i,q))=\mathcal{H}(-\vv_i,p_i)$.
Then there exist $q_1,~q_2\in\Eto$ such that the following crooked planes are pairwise disjoint:
\begin{equation*}
\CP(\vu_1,q_1),~\CP(\vu_2,q_2),~\CP(\vv_1,\g_1(q_1)),~\CP(\vv_2,\g_2(q_2)).
\end{equation*}
\end{lemma}

\end{document}